\documentclass[12pt]{article}
\usepackage[utf8]{inputenc}
\usepackage[russian,german,english]{babel}
\usepackage{microtype}
\usepackage{geometry}                
\usepackage{amssymb,amsmath,amsthm}
\usepackage{color}
\usepackage{hyperref}
\usepackage{wrapfig}
\usepackage{graphicx}
\usepackage{url}

\makeatletter
\newcommand*{\gl}{\nobreak\hskip\z@skip}
\DeclareRobustCommand*{\hyph}{\gl\hbox{-}\gl}
\DeclareRobustCommand*{\nhyph}{\gl\hbox{--}\gl}
\DeclareRobustCommand*{\mhyph}{\gl\hbox{---}\gl}
\newtheorem*{theorem}{Theorem}
\makeatother
\begin{document}
\title{Hilbert's Error?}
\author{Alexander Shen\thanks{LIRMM, CNRS \& University of Montpellier. On leave from IITP RAS (Moscow). E-mail: \url{alexander.shen@lirmm.fr, sasha.shen@gmail.com}}}
\date{}
\maketitle

\section{Hilbert's proof}

Geometric constructions have been a topic of interest of mathematicians for centuries. The Euclidean tradition uses two instruments: a straightedge (ruler) and a compass. However, limited constructions were also studied. The \emph{Mohr\nhyph Mascheroni} theorem says that every construction using both instruments can be replaced by a construction of the same object that uses only compass. (Obviously, a line cannot be drawn with a compass; we construct two different points on that line instead.) Another result, the \emph{Poncelet\nhyph Steiner} theorem, says that one application of a compass is enough: if a circle with its center is given, then every straightedge\nhyph compass construction can be performed only with a straightedge. Nowadays these results are considered more as recreational mathematics, and can be found in many books, for example, in~\cite{courant-robbins}.

Is it possible to improve the Poncelet\nhyph Steiner result and use only the circle \emph{without its center}? The same book~\cite[p.~152]{courant-robbins} provides a negative answer with a very nice argument:
\begin{quote}
\ldots if a circle, but not its center, is given, it is impossible to construct the latter by the use of the straightedge alone. To prove this we shall make use of a fact that will be discussed later: [\ldots] There exists a transformation of the plane into itself which has the following properties: (a) the given circle is fixed under the transformation, (b) Any straight line is carried into a straight line, (c) The center of the circle is carried into some other point. The mere existence of such a transformation shows the impossibility of constructing with the straightedge alone the center of the given circle. For, whatever the construction might be, it would consist in drawing a certain number of straight lines and finding their intersections with one another and with the given circle. Now if the whole figure, consisting of the given circle together with all points and lines of the construction, is subjected to the transformation whose existence we have assumed, the transformed figure will satisfy all the requirements of the construction, but will yield as result a point other than the center of the given circle. Hence such a construction is impossible.
\end{quote}
The transformation mentioned is easy to construct. Ellipse is a conic section, i.e., an image of a circle by a central projection. This projection maps the centre of the circle into a non-central point of the ellipse, so if we then apply an affine transformation to bring the ellipse back to the circle, the center is carried to some other point.

This argument goes back to great David Hilbert. It was published by Detlef Cauer~\cite{cauer} in a stronger version, saying  that for any two non-intersecting and non-concentric circles there is no way to construct their centers using only the straightedge. (In this case there is a projective transformation that preserves both circles, but not the centers, too, though it is more difficult to establish its existence.) For the case of one circle this argument, according to Cauer, was given by Hilbert in his lectures:
\begin{quote}
\foreignlanguage{german}{Nachdem ich die Ableitung eines Hilfssatzes aus der projectiven Geometrie vorausgeschickt habe (\S 1), f\"uhre ich in \S 2, auf Grund des von Herrn Hilbert f\"ur \emph{einen} Kreis angegebenen Gedankensganges, den Unm\"oglichkeitbeweis\gl$^{**}$ f\"ur zwei Kreise mit imagin\"aren Schnittpunkten im Endlichen}.\footnote{After first deriving an auxiliary proposition from projective geometry (\S 1), I provide in \S 2 the impossibility proof for the case of two circles with finite imaginary intersection points based of the approach used by Mr. Hilbert for \emph{one} circle.} 
\end{quote}
At that time Hilbert was one of the editors of \emph{Mathematische Annalen} where Cauer's paper was published. The footnote $(^{**})$ says that at the same time the proof for two circles was found by Rulf, also following Hilbert's idea. (``\foreignlanguage{german}{Den Beweis hierf\"ur hat gleichzeitig auch Herr Rulf, ebenfalls auf Grund der von Herrn Hilbert gegebenen Anregung, gefunden.}'')

\section{What is a geometric\\ construction?}

Hilbert's argument is simple, nice and convincing (so it is reproduced in many books; see, e.g.,~\cite{kac-ulam,rademacher}; in~\cite{rademacher} Cauer's proof for two circles is also presented). Still a cautious reader would ask: we are proving the non-existence of a contruction of the center of a given circle, but \emph{how do we define the notion of a geometric construction}? What kind of object does not exist?

At first, the answer looks straightforward. As Terence Tao~\cite{tao} puts it,
\begin{quote}
Formally, one can set up the problem as follows. Define a configuration to be a finite collection $\mathcal{C}$ of points, lines, and circles in the Euclidean plane. Define a construction step to be one of the following operations to enlarge the collection $\mathcal{C}$:

\begin{itemize}
\item
(Straightedge) Given two distinct points $A$, $B$ in $\mathcal{C}$, form the line $\overline{AB}$ that connects $A$ and $B$, and add it to $\mathcal{C}$.

\item
(Compass) Given two distinct points $A$, $B$ in $\mathcal{C}$, and given a third point $O$ in $\mathcal{C}$ (which may or may not equal $A$ or $B$), form the circle with centre $O$ and radius equal to the length $|AB|$ of the line segment joining $A$ and $B$, and add it to $\mathcal{C}$.

\item
(Intersection) Given two distinct curves $\gamma$, $\gamma'$ in $\mathcal{C}$ (thus $\gamma$ is either a line or a circle in $\mathcal{C}$, and similarly for $\gamma'$), select a point $P$ that is common to both $\gamma$ and $\gamma'$ (there are at most two such points), and add it to $\mathcal{C}$.
\end{itemize}

We say that a point, line, or circle is constructible by straightedge and compass from a configuration $\mathcal{C}$ if it can be obtained from $\mathcal{C}$  after applying a finite number of construction steps.
\end{quote}
Though perfectly suitable in the case of the trisection problem discussed by Tao (divide angle $A$ in a given triangle $ABC$ into three equal parts), this definition of geometric construction is obviously unsuitable for Hilbert's proof. Applying this definition literally, we see that initially we have only one object (the given circle) and we cannot apply any operation at all, so the statement becomes trivial. Definitely, Hilbert meant something else. 

There are many other regards in which the notion of a geometric construction needs to be clarified.

\begin{itemize}
\item \emph{Specifying the answer}. Is it enough if the object that we want to construct appears in $\mathcal{C}$ (see the definition above) at some step, or we should say explicitly which operation produces this object? If we want the latter, how do we deal with the non-deterministic operation of choosing an intersection point of a line and a circle? If two points $(0,0)$ and $(1,0)$ are given, it is easy to construct (using straightedge and compass) the points $(0,1)$ and $(0,-1)$, but there is no way to distinguish between these two points (for symmetry reasons).

\item \emph{Tests}. If we try to consider a geometric construction as a kind of program (which is quite natural), what kind of \emph{tests} are allowed in this program? Can we check whether some point (constructed so far) belongs to a line (also constructed already)? Can we find in which order three given points are arranged on a line (which of three points is the middle one)? If tests are not allowed, what happens if we try to find an intersection point for two parallel lines?  If tests are allowed, should we also permit \textbf{while}-loops where the number of iterations is not bounded \emph{a priori}?

\item \emph{Arbitrary points}. To solve the problem mentioned above (we do not have enough objects to start the construction), we may permit adding auxiliary points to the configuration $\mathcal{C}$. It is quite common to start a construction by saying something like ``take an arbitrary point not on line $l$''. But if we are allowed to add an arbitrary point to $\mathcal{C}$, all objects become constructible, so this is a bad idea. Instead, we may consider a game with an adversary that chooses the ``arbitrary'' point. But what are the rules of this game? Probably it should be legal to ask the adversary to choose a point that is different from the points that are already in $\mathcal{C}$, or a point that does not belong to the lines that are already in $\mathcal{C}$. What else?

\item \emph{Uniformity}. When solving, say, an angle trisection problem with straightedge and compass, do we want to have one general construction that works for every angle, or do we allow different constructions for different angles?\footnote{A similar alternative appears when we prove that 5th degree equation cannot be solved in radicals: one could prove that a general formula does not exist, or give an example of an equation with integer coefficients whose roots cannot be obtained from rational numbers by arithmetical operations and taking roots. These two settings are different.}
\end{itemize}

Now it is easy to see some problems in the argument we cited. If we require only that the construction gives a set of points that contains the center, then the projection argument does not lead to a contradiction: we know that the center is carried to a non-central point, but some other point in the set could be carried to a central point. On the other hand, if we require that the construction procedure specifies the resulting point and permits tests, we get another problem:  we cannot claim that ``the transformed figure will satisfy all the requirements of the constuction''. Indeed, after the projective transformation the test can give different results (order of points on a line may change, parallel lines may become non-parallel and vice versa). So again the argument does not work. And if we require to specify the answer point but do not permit tests, many classical constructions become illegal. For example, how do we distinguish the center of inscribed circle from the centers of three exscribed circles in a classical construction~\cite{incircle} of the center of the inscribed circle of a triangle as the intersection of the three internal angle bisectors?

These remarks look like pedantic quibbles at first, but there is a real problem here. It turned out that Cauer's result about two circles is plainly false: as shown by Arseny Akopyan and Roman Fedorov~\cite{akopyan-fedorov}, there are some pairs of disjoint non-concentric circles for which the center can be constructed by using only a straightedge. The construction, while rather long, is completely ordinary and does not involve any doubtful steps or tricks (see Section~\ref{history}).

Still, the result about one circle can be saved. Before presenting a corrected proof, we should discuss the definition of a geometric construction. The situation is similar to algorithms theory (as noted by Vladimir~A.~Uspensky on several occasions): one can construct different algorithms (like Euclid's algorithm) without defining the general notion of an algorithm, and people did this for ages. But to prove that some problem is \emph{algorithmically undecidable}, i.e., to prove that an algorithm with some properties does not exist, we need to have a formal definition of algorithm (given only in 1930s by Church, Turing, Kleene, Post and others).

So what is a geometric construction and how can we prove that one cannot construct the center of a given circle by a straightedge alone?

\section{A negative definition and\\ its positive reformulation}

Akopyan and Fedorov~\cite{akopyan-fedorov} suggested the following \emph{negative} definition of constructibility. Assume that we have some collection $\mathcal{C}$ of objects and some object $\alpha$ (say, a point). We say that $\alpha$ \emph{cannot be constructed given} $\mathcal{C}$, if there exists a collection $\mathcal{C}'$ such that (a)~$\mathcal{C}'$ contains $\mathcal{C}$, (b)~$\mathcal{C}'$ is closed under the allowed operations (taking an intersection point of two lines/circles, drawing a line through two points, etc.), (c)~$\mathcal{C}'$ does not contain $\alpha$ and (d)~the set of points in $\mathcal{C}'$ is everywhere dense (every open subset of the plane contains some point from $\mathcal{C}'$). The intuition is clear: since $\mathcal{C}'$ is dense, we may use only points from $\mathcal{C}'$ as ``arbitrary'' points, and will never construct $\alpha$ in this way.

Let us give a ``positive'' equivalent version of this definition. Consider a game where two players alternate. The game starts with the initial set $\mathcal{C}$ of given objects. The second player (Bob) adds new objects to this configuration according to requests of the first player (Alice), thus increasing the current configuration $\overline{\mathcal{C}}$. (Initially $\overline{\mathcal{C}}$ equals $\mathcal{C}$.) These requests could be of several types: (i)~add a line that goes through two  different points in $\overline{\mathcal{C}}$ (specified by Alice); (ii)~add points where two different lines/circles from $\overline{\mathcal{C}}$ (specified by Alice) intersect; (iii)~add an arbitrary point from a non-empty open set $U$ specified by Alice.\footnote{We consider constructions with the straightedge; the initial configuration may contain circles but new circles cannot be added. Allowing the use of a compass, we should add the corresponding operations both in the positive and negative definitions; they remain equivalent.} Only the moves of type~(iii) give Bob some choice; in other cases his move is uniquely determined by Alice's request. Alice wins if the object $\alpha$ (that should be constructed) appears in~$\overline{\mathcal{C}}$. Now we can say that \emph{$\alpha$ is constructible from $\mathcal{C}$} if Alice has a winning strategy in the corresponding game.

Before proving the equivalence of the ``positive'' and ``negative'' definitions, let us compare this positive definition with the questions asked in the previous section. We \emph{do not require to specify the answer} (any element of $\overline{\mathcal{C}}$ could be the answer). \emph{All tests are allowed} (Alice sees the configuration and can choose her move using full information about the configuration). Asking for \emph{arbitrary points}, Alice may request them to be as close to a point of her choice as she wishes (though, of course, she is not allowed to specify the point exactly). It gives her more power than, say, the possibility to request a point in a domain bounded by existing curves (from $\overline{\mathcal{C}}$). Note also that in four steps Alice can force Bob to add a point on a curve from $\overline{\mathcal{C}}$ (by asking for two points on both sides, a connecting line and, finally, the intersection point). Finally, the definition is given in the \emph{non-uniform setting} (each instance of the construction problem is considered in isolation, no common construction scheme is required). In this way we get a rather liberal definition\mhyph and therefore rather strong impossibility result, if we prove that the construction in the sense of this definition does not exist.

Now let us sketch the equivalence proof.\footnote{This proof essentially follows~\cite{akopyan-fedorov}, though the authors do not give explicitly a positive version of the definition speaking about ``algorithms'' instead.}

\begin{theorem}
The positive and negative definitions are equivalent.
\end{theorem}

\begin{proof}
Assume that some collection $\mathcal{C}$ and some object $\alpha$ are fixed. One direction is obvious: if there exists a collection $\mathcal{C}'$ with the properties (a)--(d) required by the negative definition, then Alice cannot have a winning strategy in the game specified in the positive definition, since Bob can maintain $\overline{\mathcal{C}}\subset \mathcal{C}'$. This is guaranteed for deterministic steps since $\mathcal{C}'$ is closed under the allowed operations. For the steps where Bob should add a point from a given open set, he uses the density of $\mathcal{C}'$.

The other direction is just a bit more complicated. Assume that a set $\mathcal{C}'$ with properties (a)--(d) does not exist. We need to prove that Alice has a winning strategy.  Let Alice start the game by forcing Bob to add four points such that any three of them are not collinear and all the lines going through pairs of these points are not parallel. For that Alice could choose arbitrary four points with these properties and ask Bob to add points nearby. If for each of Bob's choices Alice has a winning strategy in the rest of the game, then $\alpha$ is constructible in the sense of the positive definition. 

It remains to consider the case where some reply by Bob prevents Alice from winning the game. In this case we will get a contradiction with the assumption. Indeed, let $b_1,b_2,b_3,b_4$ be the points chosen by Bob in this case. No three of them are collinear and no two lines are parallel. Consider the closure of this set under the operations of adding lines (through existing points) and intersection points (of existing lines). It is easy to see that for every four points with above-mentioned properties their closure is everywhere dense in the plane. (One can use a standard trick from projective geometry and move two of four points to infinity. See also~\cite[Lemma 1]{baston-bostock}.) 

We get a bigger (and also dense) set if we add $b_1$, $b_2$, $b_3$, $b_4$ to $\mathcal{C}$ and consider the closure under all allowed operations. Denote this set by $\mathcal{C}'$. It does not contain $\alpha$ since otherwise Alice can win by forcing Bob to perform operations that give $\alpha$. Therefore, all the requirements of the negative definition are fulfilled.
\end{proof}

\section{Hilbert's proof corrected}

Now, having a rigorous definition of a geometric construction, we should repair Hilbert's proof. We still use the same idea: \emph{projective transformations could preserve the circle but move its center}. However, we need more, namely, the existence of an \emph{uncountable family} of projective transformations that preserve the circle and map the center into different points. This looks trivial (having one transformation that maps the center into a non-central point, we can then apply circle rotations), but still this is a crucial observation that differentiates the situations with one circle and two circles (where the construction is sometimes possible). 

We will prove the non-existence of a straightedge-only construction of the center of a given circle using the ``negative'' definition. Let us start with some countable family $\mathcal{G}$ of objects that includes the given circle, contains a dense set of points and is closed under the construction operations (taking a common point of two intersecting lines, etc.). Moreover, we assume that $\mathcal{G}$ is extended to the projective plane, i.e., may include infinite points and infinite line, and closure operations are allowed for infinite objects, too. For example, one may consider all ``algebraic'' objects, i.e., points, lines and circles whose coordinates (coefficients in the standard equations in a coordinate system where the given circle is defined by equation $x^2+y^2=1$) are algebraic, or just select some dense countable set of points and apply closure operations (we still have a countable set after closure).

If $\mathcal{G}$ does not contain the center of the circle, we are done (see the negative definition). If not (for example, if $\mathcal{G}$ is the set of algebraic objects) consider a projective mapping $T$ from our family that maps the center into a point outside $\mathcal{G}$. This is possible since there are uncountably many possible image points and only countably many points in $\mathcal{G}$.\footnote{We can avoid the cardinality argument by using algebraic objects and a transformation in our family that maps the center to a non-algebraic point, as it is done in~\cite{akopyan-fedorov}} Now consider the family $T^{-1}{\mathcal{G}}$ of objects whose $T$-images are in $\mathcal{G}$. (Note that finite points could be mapped by $T$ to infinite ones; this is why we needed $\mathcal{G}$ to include also infinite objects and be closed under operations on them.) By construction $\mathcal{G}'=T^{-1}{\mathcal{G}}$ does not contain the center. On the other hand, points from $\mathcal{G}'$ form a dense set and $\mathcal{G}'$ is closed under construction operations, since both properties are preserved by projective transformations.

\section{Historical remarks}\label{history}

As we have said, the proof that one cannot find the center of a circle using only a straightegde was published by Cauer~\cite{cauer} in \emph{Mathematische Annalen} and was ascribed to Hilbert (who was an editor of this journal at the time). The correction published next year did not question the validity of the proof (it dealt with the case of three circles).

The argument was since then reproduced in many popular books and textbooks (e.g.,~\cite{courant-robbins,rademacher,kac-ulam,prasolov}) without any reservations. At the same time, many authors felt the need to clarify the notion of geometric constructions. For example, Bieberbach recognized that order tests are needed to select one point among the family of points obtained by a construction, otherwise we construct only a set containing the required point, not the point itself:
\begin{quote}
\foreignlanguage{german}{Noch mag aber ausdr\"ucklich hervorgehoben werden, daß es sich in diesem Paragraphen ebenso wie bei den Poncelet\nhyph Steinerschen Konstruktionen stets um ein Konstruieren in der orientierten Ebene handelt. Es soll bei jedem gegebenen und bei jedem konstruierten Punkt feststehen, welches die Vorzeichen seiner Koordinaten sind. Anderenfalls steht nur fest, daß der gesuchte Punkt sich unter den konstruierten befindet}.~\cite[page 26]{bieberbach}\footnote{One should say explicitly that the constructions of this paragraph (as well as the Poncelet\nhyph Steiner constructions) are performed in the oriented plane. For each given and constructed point it should be determined what the signs of its coordinates are. Otherwise it is only certain that the point in question is among the constructed points.}
\end{quote}
However, the usual proof of impossibility of constructing the center of a given circle (ignoring these problems) is presented without any reservations in \S 5 of the same book.

The Cauer's argument was questioned by  C.~Gram. In~\cite{gram} he notes that Cauer's argument for two disjoint circles is equally applicable for the case when a line connecting two centers is given. Still, as Gram proves, it is enough to know one point on this line to reconstruct the centers of the circles. The reason of this paradox is explained as follows:
\begin{quote}
\ldots it is assumed that no distinction is made between proper points and lines and the points and the line at infinity; for the homology $\mathcal{H}$ may carry parallel lines into intersecting lines and conversely. Further, $\mathcal{H}$ need not preserve the (Euclidean) order of points on a line. Consequently,  if it is permitted, under the construction, to distingiush between pairs of parallel lines and pairs of intersecting lines or to decide whether or not a point, given or constructed, is between two other such points, Cauer's argument becomes invalid. It is the purpose of this note to demonstrate this by giving a construction of the centres of two given circles provided that a point of their centre line is known.
\end{quote}
Strangely, Gram did not say anything in this paper about the validity of the arguments for \emph{one} circle.

Gram's paper was published in 1956. In the next year the same observation (the possibility to construct the centers of two disjoint non-concentric circles if a point on the line connecting their centers is given) was made by A.S.~Smogorzhevskii in a popular book for advanced high school students~\cite[\S 19]{smogorzhevsky} (the Russian  edition was published in 1957). It is remarkable that in \S 18 of the same book Cauer's impossibility argument is given (with a projective transformation that preserves the circles and the line connecting their centers, but not the centers); the author made no comments on the apparent contradiction.

Much later (in 1990) the paper of V.J.~Baston and F.A.~Bostock~\cite{baston-bostock} appeared. The authors say in the abstract:
\begin{quote}
The paper points out that there are several interpretations in the literature of what is meant by a geometric construction. It is shown that these differences in interpretation are important, since certain classic results (including the Mohr\nhyph Mascheroni Theorem) are true under one but false under another.

The main aim of the paper is to reveal, in the literature, certain evidence of contradictions---contradictions originating, we feel, in the notion of what is to be understood by a construction.
\end{quote}
In the paper a difference is made between \emph{derivable} points (points that can be obtained by taking the closure of the set of given objects) and \emph{constructible} points (points that ``can be constructed''; this notion is understood intuitively, withouth any attempt to give a definition): ``The enigmatic character of the constructible points will be evident when we compare them with the mathematically precise derivable points''. The authors give examples that show that the intuitive notion of a constructible point may differ from the formal notion of a derivable point in both directions: on the rational plane $\mathbb{Q}^2$ \emph{every} point is derivable from any four points  assuming that they do not form a parallelogram and no three points are collinear (Lemma 1, p.~1019). Still intuitively the midpoint of two points is not usually considered as constructible (even if two other generic points are given; this usually is proven by saying that midpoints are not preserved by projective transformations). If we construct the midpoint of $AB$ in the rational plane by adding two auxiliary points $C$, $D$ and applying Lemma~1 mentioned above, the number of steps in the construction depends on the choice of $C$ and $D$, and we have no indication which of the points that appear during the construction is the right one.

The uniformity question is also mentioned: authors note that in the Mohr\nhyph Mascheroni construction of the intersection of lines $AB$ and $CD$ given $A$, $B$, $C$, $D$ using only a compass, the number of steps may depend on the initial configuration and be large if the intersection point is far away. In our terminology, a uniform construction algorithm may need a loop.

As to a more formal definition of a geometric construction, one should mention the paper of Yury Manin~\cite{manin} in the collection of articles written for advanced high school students. Manin introduces the closure operations (operations 1--3 on p.~208) and notes the necessity of ``arbitrary'' auxiliary points chosen in the regions bounded by existing objects (p.~209). However, the game setting is not mentioned; it is said instead that ``a straightedge and compass construction is a sequence that consists of finitely many steps of the described type''.
(\emph{\foreignlanguage{russian}{Построением с помощью циркуля и линейки называется последовательность, состоящая из конечного числа описанных шагов.}})
This formulation, it seems, implicitly assumes that the number of steps does not depend on the choice of ``arbitrary'' points (this could create problems with the Mohr\nhyph Mascheroni theorem). Manin speaks about straightedge and compass constructions (where a dense set of points can be constructed as soon as two points are given, so choosing arbitrary points is almost unnecessary) and does not say anything about straightedge\hyph only constructions.

Erwin Engeler applies the general framework of computations in first order structure developed in~\cite{engeler-general} to elementary geometry~\cite{engeler-geometry}. He recognizes the need for tests and explicitly lists allowed types of tests. However, he does not consider the auxiliary points problem; instead, he assumes that points $(0,0)$, $(0,1)$, $(1,0)$ are always given for free. He does not consider straightedge\hyph only constructions.

Akopyan and Fedorov~\cite{akopyan-fedorov} suggest the ``negative'' definition of constructibility and show how to correct Hilbert's argument using this negative definition (and the family of projective transformations). They also state that the negative definition is equivalent to an existence of a ``construction algorithm'' but do not even attempt to define the latter notion. They not only show that Cauer's proof has problems (this was mentioned already in~\cite{gram}), but also answer a question posed in~\cite{gram} and provide an example of two disjoint non-concentric circles such that the centers \emph{can} be constructed. (They noted that Poncelet configuration for some pairs of circles provides a symmetric configuration, and then the center line can be constructed and Gram's argument can be applied to find the centers.) They also provide a proof (based on their negative definition) that for some other pair of circles (also given explicitly) the centers \emph{cannot} be constructed. The plan could be described as follows: if it were always possible, Poncelet\nhyph Steiner construction could be used to construct also the common tangents, and therefore this can be done for every two conic sections, but sometimes this problem gives an irreducible cubic equation that cannot be solved with a straightedge and compass.

It seems that Cauer may already have a similar argument in mind when he wrote:
\begin{quote}
\ldots \foreignlanguage{german}{Ich benutze diese Gelegenheit, um auf die algebraische Behandlung dieser Probleme hinzuweisen. W\"are es m\"oglich, bei zwei Kreisen mit imagin\"aren Schnittpunkten im Endlichen die Mittelpunkte mit dem Lineal allein zu finden, so k\"onnte man auch bei zwei Kegelschnitten, von denen mindestens der eine ganz ausgezogen gedacht wird, und die vier imagin\"are Schnittpunkte haben, deren reelle Verbindungslinien mit dem Lineal allein finden; dies Problem f\"uhrt aber bekanntlich auf eine irreduzible Gleichung dritten Grades, deren Wurzeln nach der Theorie der algebraischen Gleichungen nicht durch Aufl\"osung von quadratischen und linearen Gleichungen gefunden werden k\"onnen. Das, worauf es mir damals ankam, war: zu zeigen, da\ss\ man mit Benutzung des Hilbertschen Gedankenganges ohne diese Theorie ganz elementar zum Ziel kommt. Diese liefert aber, wie Herr Schur bemerkte, das weitere Resultat, da\ss\ auch bei zwei Kreisen mit reellen Schnittpunkten, wenn nur der \emph{eine} als bekannt angesehen wird, die Konstruktion der Mittelpunkte mit dem Lineal allein unm\"oglich ist; denn auch diesem Problem entspricht eine irreduzible Gleichung dritten Grades}.~\cite[Berichtigung, Bd.~74]{cauer}\footnote{I use this opportunity to mention also the algebraic treatment of this problem. If it were possible to find the centers of two circles with finite imaginary intersection points using only a straightedge, then one could also, using a straightedge alone, find the real connecting lines of two conic sections of which at least one is thought of as fully extended and who have four imaginary intersection points. However, this problem is known to lead to an irreducible equation of third degree whose roots cannot be obtained by solving linear and quadratic equations, as the theory of algebraic equations says. My main concern at that time was to show that one can obtain the result without using this theory, by an elementary argument outlined by Hilbert. However, the algebraic reasoning gives (as noted by Mr.~Schur) an additional result: for two circles with real intersection points one cannot construct the centers using only a straightedge, if only \emph{one} circle is assumed to be known, since this problem also corresponds to an irreducible equation of third degree.}
\end{quote}
It is not clear, however, whether we may interpret `Verbindungslinien' as common tangents (used in the argument suggested in~\cite{akopyan-fedorov}) and what is the exact meaning of `ganz ausgezogen' or `als bekannt angesehen'. May be the Schur's remark should be understood as follows: if one of the intersecting circles is erased in the small neighborhoods of the intersection points, we cannot construct their centers anymore (using only a straightedge).

Finally, let us note that in Hilbert's famous book on the foundations of geometry~\cite[\S 36]{hilbert-grundlagen} there are some remarks that can be understood as the definition of constructibility in logical terms: a point with some properties is constructible with restricted means if its existence can be proven using a specified set of axioms. However, Hilbert does not give exact definitions and statements, and it is not clear whether this approach can be applied to, say, compass-only constructions.

\section*{Acknowledgements}
I am grateful to Arseny Akopyan and Roman Fedorov for sharing their paper~\cite{akopyan-fedorov} and its preliminary versions, and for discussing their results; to Sergey Markelov who participated in these discussions and had a lot of interesting suggestions; to Serge Lvovski for critical remarks; to Rupert H\"olzl who corrected the translation of German quotations and made many useful remarks. The author was supported by RaCAF ANR-15-CE40-0016-01  and RBFR 16-01-00362
grants while working on this paper.


\begin{thebibliography}{9}
\bibitem{courant-robbins} Richard Courant, Herbert Robbins. Revised by Ian Stewart. \emph{What is Mathematics? An elementary approach to ideas and methods}. Oxford University Press, 1996.

\bibitem{cauer}
Detlef Cauer, Über die Konstruktion des Mittelpunktes eines Kreises mit dem Lineal allein. \emph{Math. Ann.} Bd. 73, S. 90--94. A correction (Berichtigung): Bd. 74, S. 462--464 (1913).

\bibitem{rademacher}
Hans Rademacher, Otto Toeplitz, \emph{Von Zahlen und Figuren. Proben mathematischen Denkens f\"ur Liebhaber der Mathematik ausgew\"ahlt und dargestellt}. Zweite Augflage. Springer-Verlag, Berlin, Heidelberg, 1933.

\bibitem{kac-ulam}
Mark Kac, Stanislaw M. Ulam, \emph{Mathematic and Logic}, Dover publications, 1992, copyright: Encyclopaedia Britannica, 1968

\bibitem{tao}
\url{https://terrytao.wordpress.com/2011/08/10/a-geometric-proof-of-the}\\ \url{-impossibility-of-angle-trisection-by-straightedge-and-compass/} 

\bibitem{incircle}Incircle and excircles of a triangle, wikipedia article, \url{https://en.wikipedia.org/wiki/Incircle_and_excircles_of_a_triangle}

\bibitem{akopyan-fedorov}
Arseny Akopyan, Roman Fedorov, \emph{Two circles and only a straightedge}, \url{https://arxiv.org/abs/1709.02562}

\bibitem{baston-bostock}
V.J.~Baston, F.A.~Bostock, On the impossibility of ruler-only constructions, \emph{Proceedings of the American Mathematical Society}, \textbf{110} (4), December 1990.

\bibitem{bieberbach} Ludwig Bieberbach, \emph{Theorie der geometrischen Konstruktionen}, Springer Basel AG, 1952, \url{DOI 10.1007/978-3-0348-6910-2}

\bibitem{prasolov}
Viktor V.~Prasolov. Solid Geometry Problems. Moscow, MCCME Publishers, 2010, in Russian.
(\foreignlanguage{russian}{В.\,В.\,Прасолов. Задачи по стереометрии. Москва, МЦНМО, 2010}.)
See also: Viktor \emph{Prasolov, Problems in Plane and Solid Geometry}, translated and edited by Dmitry Leites, available at \url{http://students.imsa.edu/~tliu/Math/planegeo.pdf} (October 2017), chapter 30.

\bibitem{gram}
Christian Gram, A remark on the construction of the centre of a circle by means of the ruler, \emph{Math. Scand.}, 4 (1956), 157--160.

\bibitem{smogorzhevsky}
A.S.~Smogorzhevskii, \emph{The Ruler in Geometrical Constructions}, Pergamon Press, 1961.
\foreignlanguage{russian}{Русский оригинал: Александр Степанович Смогоржевский, \emph{Линейка в геометрических построениях} (Популярные лекции по математике), Москва, Государственное издательство технико\hyph теоретической литературы, 1957.}

\bibitem{manin}
Yuri I. Manin, On the solvability of construction problems by compass and straightedge, \emph{The Encyclopedia of Elementary Mathematics, volume IV, Geometry}. Moscow: Fizmatgiz, 1963, p.~205--228. In Russian. 
(\foreignlanguage{russian}{Ю.\,И.\,Манин. О разрешимости задач на построение с помощью циркуля и линейки. \emph{Энциклопедия элементарной математики. Том IV. Геометрия.} Москва: Государственное издательство физико\hyph математической литературы, 1963.})

\bibitem{engeler-general}
Erwin Engeler, Algorithmic Properties of Structures, \emph{Mathematical Systems Theory}, \textbf{1}, 183--195.
 
\bibitem{engeler-geometry}
Erwin Engeler, Remarks on the theory of geometrical constructions, \emph{The Syntax and Semantics of Infinitary Languages}, Lecture Notes in Mathematics, v.~72, p.~64--76.

\bibitem{hilbert-grundlagen}
David Hilbert, \emph{Grundlagen der Geometrie}, 7th edition, Leipzig, Berlin: Teubner Verlag, 1930.

\end{thebibliography}
\end{document}